\documentclass[a4paper,12pt]{amsart}
\def\norm{\operatorname{norm}\nolimits}
\def\pol{\operatorname{pol}\nolimits}
\def\BF{{\mathbf{F}}}
\def\BZ{{\mathbf{Z}}}
\def\CI{{\mathcal{I}}}
\def\CJ{{\mathcal{J}}}
\def\Hom{\operatorname{Hom}\nolimits}

\def\Ind{\operatorname{Ind}\nolimits}
\def\Ext{\operatorname{Ext}\nolimits}
\def\Res{\operatorname{Res}\nolimits}

\def\res{\operatorname{res}\nolimits}
\def\iso{\buildrel \sim\over\to}
\newtheorem{thm}{Theorem}[section]
\newtheorem{lemma}[thm]{Lemma}
\newtheorem{cor}[thm]{Corollary}

\theoremstyle{definition}
\newtheorem{rem}[thm]{Remark}

\begin{document}
\title{Finite generation of cohomology of finite groups}
\author{Rapha\"el Rouquier}
\address{Department of Mathematics, UCLA, Box 951555,
Los Angeles, CA 90095-1555, USA}
\email{rouquier@math.ucla.edu}
\thanks{The author was partially supported by the NSF grant DMS-1161999.}

\date{October 2014}
\maketitle

\begin{abstract}
We give a proof of the finite generation of the cohomology ring of a finite
$p$-group over $\BF_p$
by reduction to the case of elementary abelian groups, based on
Serre's Theorem on products of Bocksteins.
\end{abstract}

\section{Definitions and basic properties}
Let $k=\BF_p$. Given $G$ a finite group, we put
$H^*(G)=H^*(G,k)$. We refer to \cite{Ev} for results on group cohomology.

Given $A$ a ring and $M$ an $A$-module, we say that $M$ is finite over $A$
if it is a finitely generated $A$-module.

\smallskip
Let $G$ be a finite group and $L$ a subgroup of $G$. 
We have a restriction map $\res_L^G:H^*(G)\to H^*(L)$. It gives
$H^*(L)$ the structure of an $H^*(G)$-module.

We denote by $\norm_L^G:H^*(L)\to H^*(G)$ the norm map. If $L$ is central
in $G$, then we have $\res_L^G\norm_L^G(\xi)=\xi^{[G:L]}$ for all $\xi\in H^*(L)$.

When $L$ is normal in $G$,
we denote by $\inf^G_{G/L}:H^*(G/L)\to H^*(G)$ the inflation map.

\smallskip
Let $E$ be an elementary abelian $p$-group.
The Bockstein $H^1(E)\to H^2(E)$ induces an injective
morphism of algebras $S(H^1(E))\hookrightarrow H^*(E)$.
We denote by $H^*_{\pol}(E)$ its image.
Note that $H^*(E)$ is a finitely generated $H^*_{\pol}(E)$-module and
given $\xi\in H^*(E)$, we have $\xi^p\in H^*_{\pol}(E)$.

\section{Finite generation for finite groups}

The following result is classical. We provide here a proof independent
of the finite generation of cohomology rings.

\begin{lemma}
\label{le:onto}
Let $G$ be a $p$-group and $E$ an elementary abelian subgroup.
Then, $H^*(E)$ is finite over $H^{\mathrm{even}}(G)$.
\end{lemma}

\begin{proof}
The result is straightforward when $G$ is elementary abelian. As a consequence,
given $G$,
it is enough to prove the lemma when $E$ is a maximal elementary abelian
subgroup.
We prove the lemma by induction on $|G|$. Let $Z\le Z(G)$ with $|Z|=p$.
Let $P$ be a complement to $Z$ in $E$. Let $A=
\mathrm{inf}_{E/Z}^E(H^*_{\pol}(E/Z))$. Let $x$ be a generator of
$H^2(Z)\iso H^2(E/P)$ and $y=\inf_{E/P}^E(x)$.
We have
$$H^*_{\pol}(E)=A\otimes k[y].$$

Let $\xi=\res_E^G(\norm_Z^G(x)^p)$.
We have $\res_Z^E(\xi)=x^{p[G:Z]}$, so $\xi-y^{p[G:Z]}\in H^*_{\pol}(E)\cap
\ker\res_Z^E= A^{>0}H^*_{\pol}(E)$.
We deduce that $H^*(E)$ is finite over its subalgebra generated
by $A$ and $\xi$. 

By induction, $H^*(E/Z)$ is finite over $H^*(G/Z)$. We deduce that
$H^*(E)$ is finite over its subalgebra generated by $\xi$ and 
$\mathrm{\inf}_{E/Z}^E\res_{E/Z}^{G/Z}H^*(G/Z)=
\res_E^G\mathrm{inf}_{G/Z}^GH^*(G/Z)$.
\end{proof}

Let us recall a form of Serre's Theorem on product of Bocksteins
\cite{Se}. We state the result over the integers for a useful consequence
stated in Corollary \ref{co:generation}.

\begin{thm}[Serre]
\label{th:Serre}
Let $G$ be a finite $p$-group.
Assume $G$ is not elementary abelian. Then, there is $n\ge 2$, there are
subgroups $H_1,\ldots,H_n$ of index $p$ of $G$ and an exact
sequence of $\BZ G$-modules
$$0\to \BZ\to \Ind_{H_n}^G\BZ\to\cdots\to\Ind_{H_1}^G \BZ\to \BZ\to 0$$
defining a zero class in $\Ext^n_{\BZ G}(\BZ,\BZ)$.
\end{thm}

\begin{proof}
Serre shows there are elements $z_1,\ldots,z_m\in H^1(G,\BZ/p)$ such that
$\beta(z_1)\cdots\beta(z_m)=0$. The element $z_i$ corresponds to
a surjective morphism $G\to\BZ/p$ with kernel $H_i$, and we identify
$\Ind_{H_i}^G\BZ$ with $\BZ[G/H_i]=\BZ[\sigma]/(\sigma^p-1)$, where
$\sigma$ is a generator of $G/H_i$.
The element $\beta(z_i)\in H^2(G,\BZ/p)$
is the image of the class $c_i\in H^2(G,\BZ)$ given by the exact sequence
$$0\to \BZ\xrightarrow{1+\sigma+\cdots+\sigma^{p-1}}
\Ind_{H_i}^G\BZ\xrightarrow{1-\sigma}
\Ind_{H_i}^G\BZ\xrightarrow{\text{augmentation}}\BZ\to 0.$$
Let $c=c_1\cdots c_m\in H^{2m}(G,\BZ)$.
The image of $c$ in $H^{2m}(G,\BZ/p)$ vanishes, hence $c\in pH^{2m}(G,\BZ)$.
Fix $r$ such that $|G|=p^r$. 
Since $|G|H^{>0}(G,\BZ)=0$, we deduce that $c^r=0$.
\end{proof}

We will only need the case $R=\BF_p$ of the corollary below.
We denote by $D^b(RG)$ the
derived category of bounded complexes of finitely generated $RG$-modules.

\begin{cor}
\label{co:generation}
Let $G$ be a finite group and $R$ a discrete valuation ring with residue
field of characteristic $p$ or a field of charactetistic $p$. Assume
$x^{p-1}=1$ has $p-1$ solutions in $R$. 

Let
$\CI$ be the thick
subcategory of $D^b(RG)$ generated by modules of the form
$\Ind_E^G M$, where $E$ runs over elementary abelian
subgroups of $G$ and $M$ runs over one-dimensional
representations of $E$ over $R$.

We have $\CI=D^b(RG)$.
\end{cor}

\begin{proof}
Assume first $G$ is an elementary abelian $p$-group.
Let $L$ be a finitely generated
$RG$-module. Consider a projective cover $f:P\to L$ and let
$L'=\ker f$. The $R$-module
$L'$ is free, so $L'$ is an extension of 
$RG$-modules that are free of rank $1$ as $R$-modules. So $L'\in\CI$ and
similarly $P\in\CI$, hence $L\in\CI$. As a consequence, the corollary
holds for $G$ elementary abelian.

Assume now $G$ is a $p$-group that is 
not elementary abelian. We proceed by induction on $|G|$. 
Let $L$ be a finitely generated $RG$-module. By induction,
$\Ind_H^G\Res_H^G(L)\in\CI$ whenever $H$ is a proper subgroup of $G$.
Applying $L\otimes_{\BZ G}-$ to the exact sequence of Theorem 
\ref{th:Serre}, we obtain an exact sequence
$$0\to L\to \Ind_{H_n}^G\Res_{H_n}^G(L)\to\cdots\to\Ind_{H_1}^G \Res_{H_1}^G
(L)\to L\to 0.$$
Since that sequence defines the zero class in $\Ext^n(L,L)$,
it follows that $L$ is a direct summand of 
$0\to \Ind_{H_n}^G\Res_{H_n}^G(L)\to\cdots\to\Ind_{H_1}^G \Res_{H_1}^G
(L)\to 0$ in $D^b(RG)$. We deduce that $L\in\CI$.

Finally, assume $G$ is a finite group. Let $P$ be a Sylow $p$-subgroup of $G$
and let $L$ a finitely generated $RG$-module. 
We know that $\Ind_P^G\Res_P^G(L)\in\CI$.
Since $L$ is a direct summand of $\Ind_P^G\Res_P^G(L)$, we deduce that
$L\in\CI$.
\end{proof}

\begin{rem}
Corollary \ref{co:generation} implies a corresponding generation result for
the stable
category of $RG$. That was observed by the author in the mid 90s and 
communicated to J.~Carlson who wrote an account of this in \cite{Ca}.
\end{rem}

\begin{thm}[Golod, Venkov, Evens]
Let $G$ be a finite $p$-group.
The ring $H^*(G)$ is finitely generated. Given
$M$ a finitely generated $kG$-module, then
$H^*(G,M)$ is a finitely generated $H^*(G)$-module.
\end{thm}

Note that the case where $G$ is an arbitrary finite group follows easily, cf
\cite{Ev}.

\begin{proof}
Let $S$ be a finitely generated subalgebra of
$H^{\mathrm{even}}(G)$ such that $H^*(E)$ is a finitely generated
$S$-module for every elementary abelian subgroup $E$ of $G$. Such an
algebra exists by Lemma \ref{le:onto}.

\smallskip
Let $\CJ$ be the full subcategory of $D^b(kG)$
of complexes $C$ such that the $S$-module
$H^*(G,C)=\bigoplus_i\Hom_{D^b(kG)}(k,C[i])$ is
finitely generated.

Let $C_1\to C_2\to C_3\rightsquigarrow$ be a distinguished triangle in
$D^b(kG)$. We have a long exact sequence
$$\cdots\to H^i(C_1)\to H^i(C_2)\to H^i(C_3)\to H^{i+1}(C_1)\to\cdots$$
Assume $C_1,C_3\in\CJ$. Let $I$ be a finite generating set
of $H^*(C_1)$ as an $S$-module and $J$ a finite generating set
of $\ker(H^*(C_3)\to H^{*+1}(C_1))$ as an $S$-module. Let $I'$ be the image
of $I$ in $H^*(C_2)$ and let $J'$ be a finite subset of $H^*(C_2)$ with
image $J$. Then, $I'\cup J'$ generates $H^*(C_2)$ as an $S$-module, hence
$C_2\in\CJ$.

Note that if $C\oplus C'\in\CJ$, then $C\in\CJ$.
We deduce that $\CJ$ is a thick subcategory of $D^b(kG)$.

\smallskip
Let $E$ be an elementary abelian subgroup of $G$.
Since $H^*(G,\Ind_E^G(k))\simeq H^*(E,k)$ is a finitely generated $S$-module,
we deduce that $\Ind_E^G(k)\in\CJ$.

We deduce from Corollary \ref{co:generation} that $\CJ=D^b(kG)$.
\end{proof}

\end{document}